\documentclass[letterpaper,12pt]{amsart}

\title{L-infinity optimization to linear spaces and phylogenetic trees}

\author{Daniel Irving Bernstein}
\author{Colby Long}
\address{Department of Mathematics \\ North Carolina State University, Raleigh, NC 27695}
\email{dibernst@ncsu.edu}
\email{celong2@ncsu.edu}

\usepackage{latexsym,array,delarray,amsthm,amssymb,epsfig,graphicx,multirow}%eufrak
\usepackage{algorithmic}
\usepackage{algorithm}
\usepackage{amsmath}
\usepackage[numbers]{natbib}
\usepackage{tikz}
\usepackage{verbatim}
\usepackage{graphicx}
\usepackage{url}

\theoremstyle{plain}
\newtheorem{thm}{Theorem}[section]
\newtheorem{lemma}[thm]{Lemma}
\newtheorem{prop}[thm]{Proposition}
\newtheorem{cor}[thm]{Corollary}

\newtheorem*{thm*}{Theorem}
\newtheorem*{lemma*}{Lemma}
\newtheorem*{prop*}{Proposition}
\newtheorem*{cor*}{Corollary}
\newtheorem*{conj*}{Conjecture}

\theoremstyle{definition}
\newtheorem{defn}[thm]{Definition}
\newtheorem*{defn*}{Definition}
\newtheorem{ex}[thm]{Example}
\newtheorem{pr}[thm]{Problem}

\theoremstyle{remark}

\usepackage{thmtools, thm-restate}

\newcommand{\rr}{\mathbb{R}}

\newcommand{\calm}{\mathcal{M}}

\newcommand{\calt}{\mathcal{T}}

\newcommand{\ind}{\mbox{$\perp \kern-5.5pt \perp$}}

\newcommand{\inv}{^{-1}}
\newcommand{\rank}{\textnormal{rank}}

\newcommand{\supp}{\textnormal{supp}}
\newcommand{\cone}{\textnormal{cone}}

\newcommand{\type}{\ensuremath\textnormal{type}}
\newcommand{\sign}{\ensuremath\textnormal{sign}}

\tikzstyle{vertex}=[circle, draw, inner sep=0pt, minimum size=6pt, fill=black]
\newcommand{\vertex}{\node[vertex]}

\begin{document}
\maketitle

\begin{abstract}

Given a distance matrix consisting of pairwise
distances between species, a distance-based phylogenetic
reconstruction method returns a tree metric or equidistant tree metric (ultrametric) that best fits the data. 
We investigate distance-based phylogenetic reconstruction
using the $l^\infty$-metric. In particular, we analyze 
the set of $l^\infty$-closest ultrametrics and tree metrics
to an arbitrary dissimilarity map to determine its
dimension and the tree topologies it represents.
In the case of ultrametrics, we decompose the space of dissimilarity maps on 3 elements and on 4 elements relative to the tree topologies represented.

Our approach is to first address uniqueness 
issues arising in $l^\infty$-optimization to linear spaces.
We show that the $l^\infty$-closest point in a linear space is unique if and only if the underlying matroid of the linear space is uniform.  We also give a polyhedral decomposition of $\rr^m$ based on the dimension of the set of $l^\infty$-closest points in a linear space.

\end{abstract}

\section{Introduction}

One approach to phylogenetic reconstruction is to use \emph{distance-based} methods. 
Given a distance matrix consisting of the pairwise
distances between $n$ species, a distance-based method
returns a tree metric or equidistant tree metric (ultrametric)
that best fits the data. Typically, the distance matrix
is constructed from biological data.
It has been shown that both the set of equidistant tree metrics
and the set of tree metrics have close connections to 
tropical geometry \cite{ardila2004,ardila-klivans2006,Speyer}.
Because addition in the tropical semiring is defined as taking the maximum of two elements, the $l^\infty$-metric offers an 
appealing choice as a measure of best fit for phylogenetic reconstruction.

Computing a closest tree metric to a given distance matrix 
using the $l^\infty$-metric is NP-hard \cite{agarwala1998}.
However, there exists a polynomial-time algorithm for computing an $l^\infty$-closest \emph{equidistant} tree metric \cite{chepoi}.
Although the algorithm gives us a way to compute
 a closest equidistant tree metric to an arbitrary point in
 $\rr^{\binom{n}{2}}$ quickly, the set of closest equidistant tree
 metrics is not in general a singleton. Indeed, it may
 be of high dimension or contain points corresponding to 
 trees with entirely different topologies. Thus, for phylogenetic
 reconstruction, there may be several different trees that explain the data equally well from the perspective of the $l^{\infty}$-metric. Recent work has studied the 
properties of equidistant tree space with the $l^\infty$-metric \cite{ardila2004,lin-sturmfels2016,lin-yoshida2016} but
to our knowledge the dimensions and topologies of the sets
of $l^\infty$-closest equidistant tree metrics have not been examined. 
Similarly, one might ask all of the same questions 
for tree metrics. 
Thus, we are motivated by the following problem.

\begin{pr}
\label{pr:phylogenetics}
	Given a dissimilarity map $x \in \rr^{\binom{n}{2}}$,
	describe the set of (equidistant) tree metrics that are closest to $x$ in the $l^\infty$-metric.
\end{pr}

For both equidistant tree metrics and tree metrics,
 we obtain results concerning 
the dimensions of these sets
as well as the tree topologies involved.
Since the set of tree metrics and the set of 
equidistant tree metrics on $n$ species are 
both polyhedral complexes, 
we begin by addressing the following problem as a stepping stone. The results obtained may be of independent interest
to those studying combinatorics or optimization.

\begin{pr}\label{pr:linearspace}
	Given a point $x \in \rr^m$ and a linear space $L \subseteq \rr^m$,
	describe the subset of $L$ consisting of points that are closest to $x$ in the $l^\infty$-metric.
\end{pr}

Just as with tree metrics, 
the $l^\infty$-closest point in a linear space is not unique in general.
We give a polyhedral decomposition of $\rr^m$
based on the dimension of the set of points in $L$ that are $l^\infty$-closest to $x$.
One particularly nice implication of this decomposition is the following.

\begin{restatable*}{Theorem}{linearunique}
\label{thm:linearunique}
	Let $L \subseteq \rr^m$ be a linear space.
	Then the $l^\infty$-closest point to $x$ in $L$ is unique for all $x \in \rr^m$
	if and only if the matroid underlying $L$ is uniform.
\end{restatable*}

The set of (equidistant) tree metrics on a fixed set of species is a polyhedral fan.
Each open cone in this fan is the set of (equidistant) tree metrics corresponding to a particular tree topology.
For many dissimilarity maps,
optimizing to the set of (equidistant) tree metrics will be equivalent 
to optimizing to the
linear hull of one such maximal cone.
The equations defining the linear hulls of these cones
are highly structured and the corresponding matroids are not uniform. Therefore, Theorem \ref{thm:linearunique} implies the existence
of dissimilarity maps with a positive-dimensional set of $l^\infty$-closest
(equidistant) tree metrics.
For example, we show that there is a full-dimensional set of dissimilarity maps in $\rr^{\binom{n}{2}}$ for
which the set of $l^\infty$-closest 
equidistant tree metrics has dimension $n -2$.
Our construction shows that we can often obtain many $l^\infty$-closest equidistant tree metrics to a dissimilarity map by adjusting
branch lengths in an equidistant tree representing one such metric. 
We will also see that there are dissimilarity maps for which the set of $l^\infty$-closest equidistant tree metrics
contains equidistant tree metrics representing
different tree topologies.
In the case of $4$-leaf trees, we provide a 
decomposition of $\mathbb{R}^{\binom{4}{2}}$ according to the
topologies represented.

We begin in Section \ref{sec:linearspaces}  with our results on $l^\infty$-optimization to linear spaces.
In particular, we give a natural way to assign a combinatorial type to each $x \in \rr^m$ with respect to some linear subspace $L \subseteq \rr^m$.
We show that this combinatorial type gives a polyhedral decomposition of $\rr^m$ based on the dimension
of the set of $l^\infty$-closest points in $L$, from which
Theorem \ref{thm:linearunique} follows.
Section \ref{sec:phylogenetics} applies the results and ideas
for linear spaces to phylogenetics.
We investigate questions that would be of practical interest for phylogenetic reconstruction such as the dimension and
corresonding tree topologies in the set of closest ultrametrics.
We conclude by exploring the $l^\infty$-metric as a distance-based method for reconstructing tree metrics.

\section{$l^\infty$-optimization to Linear Spaces}
\label{sec:linearspaces}
Given a linear space $L \subseteq \rr^m$,
we demonstrate a way to associate a sign vector in 
$\{+,-,0\}^m$ to each $x \in \rr^m$.
The associated sign vectors are then precisely
the elements of the oriented matroid associated to $L$.
For each $x \in \rr^m$, this vector will encode information about the dimension of the set of $l^\infty$-closest points to $x$ in $L$.
\\
\indent
We begin this section by reviewing the necessary background from oriented matroid theory.
More details can be found in \cite[Ch. 6 and 7]{ziegler2000}.

\subsection{Background on Oriented Matroids}
For any real number $r\in \rr$, $\sign(r) \in \{+,-,0\}$ is the
\emph{sign of $r$}.
For a linear functional $c \in (\rr^m)^*$,
$\sign(c) \in \{+,-,0\}^m$ is defined by $\sign(c)_i = \sign(c_i)$.
Given a sign vector $\sigma \in \{+,-,0\}^m$, we define $|\sigma| := \#\{i: \sigma_i \neq 0\}$.
For a linear space $L \subseteq \rr^m$ the \emph{oriented matroid associated to $L$}, denoted $\mathcal{O}_L$,
is the set of all sign vectors $s$ in $\{+,-,0\}^m$  such that 
$s = \sign(c)$
for some linear functional $c \in (\rr^m)^*$ that vanishes on $L$.

The elements of an oriented matroid $\mathcal{O}$ are the \emph{signed vectors} of $\mathcal{O}$.
Let $\prec^*$ be the partial order on $\{+,-,0\}$ given by $0 \prec^* +$ and $0 \prec^* -$ with $+$ and $-$ incomparable.
Then $\prec$ is the partial order on $\{+,-,0\}^m$ that is the cartesian product of $\prec^*$ $m$ times.
The \emph{signed circuits} of an oriented matroid $\mathcal{O}$ are the signed vectors of $\mathcal{O}$ that are minimal with respect to $\prec$.

An oriented matroid can also be derived from a \emph{zonotope},
the image of a cube under an affine map.
Let $C_\delta(x) \subseteq \rr^m$ denote the cube of side length $2\delta$ centered at $x$.
That is,
\[
	C_\delta(x) = \{y \in \rr^m : |y_i - x_i| \le \delta, i = 1, \dots, m\}.
\]
To each face $F$ of $C_\delta(x)$, associate a sign vector $\sign(F) \in \{+,-,0\}^m$ as follows
\[
	\sign(F)_i = \begin{cases}
		+ & \text{ if } y_i = x_i + \delta \quad \text{for all } y \in F \\
		- & \text{ if } y_i = x_i - \delta \quad \text{for all } y \in F \\
		0 & \text{ otherwise}.
	\end{cases}
\]

\noindent Figure \ref{fig:type2d} gives an illustration of the sign vectors associated to a square.

\begin{figure}[h!]
	\centering
	\begin{tikzpicture}
	\node (v) at (0,0)[label=center:{$(0,0)$}]{};
	\vertex (v) at (45:1) [label=45:{$(+,+)$}]{};
	\vertex (x) at (135:1) [label=135:{$(-,+)$}]{};
	\vertex (y) at (225:1) [label=225:{$(-,-)$}]{};
	\vertex (z) at (315:1) [label=315:{$(+,-)$}]{};
	\path
		(v) edge node[above]{$(0,+)$} (x)
		(x) edge node[left]{$(-,0)$} (y)
		(y) edge node[below]{$(0,-)$} (z)
		(z) edge node[right]{$(+,0)$} (v)
	;
	\end{tikzpicture}
	\caption{Sign vectors corresponding to faces of a square.}\label{fig:type2d}
\end{figure}
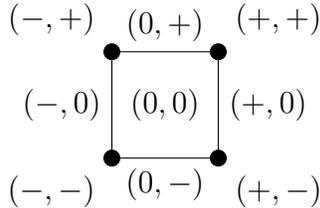

Let $V \in \rr^{(m - d)\times m}$ be a matrix of full row rank and let
$\pi: \rr^m \rightarrow \rr^{m - d}$ be the
affine map given by $x \mapsto  Vx$. 
For fixed $x \in \rr^m$ and $\delta > 0$, 
$\pi(C_\delta(x)) \subset \rr^{m - d}$ is a polytope called the
\emph{zonotope associated to $V$}.
The inverse image of each face of the zonotope is a face of  $C_\delta(x)$. Thus, for each face $G$ of $\pi(C_\delta(x))$,
we define $\sign(G) := \sign(\pi^{-1}(G)).$
The collection of all such sign vectors is an oriented matroid which only depends on the matrix $V$
and so we denote it $\mathcal{O}_V$.
The following proposition relates oriented matroids from zonotopes to oriented matroids from linear spaces.
\begin{prop}[\cite{ziegler2000},Corollary 7.17]\label{linearzonotopes}
	Let $V \in \rr^{(m - d)\times m}$ be a matrix of full row rank.
	Then $\mathcal{O}_V = \mathcal{O}_{\ker V}$.
\end{prop}

Given an oriented matroid $\mathcal{O} \subseteq \{+,-,0\}^m$ and $\sigma \in \mathcal{O}$,
we define the \emph{support of $\sigma$}, denoted $\supp(\sigma)$, to be the set of indices of $\sigma$ that are nonzero.
That is, $\supp(\sigma) := \{i \in \{1,\dots,m\}: \sigma_i \neq 0\}$.
Then the collection of subsets of $\{1,\dots, m\}$
that are supports of elements of $\mathcal{O}$ and minimal with respect to inclusion form
the circuits of a matroid, denoted $\calm_{\mathcal{O}}$.
We call this the \emph{matroid underlying $\mathcal{O}$}.
%In symbols, $\calm_{\mathcal{O}} := \{\supp(\sigma) : \sigma \in \mathcal{O}\}$.
When $\mathcal{O}$ is associated to a linear space $L$, that is $\mathcal{O} = \mathcal{O}_L$,
we simplify notation and write $\calm_L$ instead of $\calm_{\mathcal{O}_L}$.
For more background on matroids, see \cite{oxley2011}.

\subsection{$l^\infty$-Optimization to Linear Spaces}

In the rest of this section, we will use the language of matroids to state our main results for linear spaces. Before we begin, we establish some notation that will be used throughout the
entire paper.

\begin{defn}
	Let $S \subseteq \rr^m$ be an arbitrary set and let $x,z \in \rr^m$.
	We denote the $l^\infty$-distance from $x$ to $z$ by $d(x,z)$,
	the $l^\infty$-distance from $x$ to $S$ by $d(x,S)$
	and the set of all points in $S$ closest to $x$ by $C(x,S)$.
	That is
	\[
		d(x,z) := \sup_{i} |x_i-z_i| \qquad d(x,S) := \inf_{y \in S} d(x,y) \qquad C(x,S) := \{y \in S : d(x,y) = d(x,S)\}.
	\]
\end{defn}

Note that $C(x,S) = C_{d(x,S)}(x) \cap S$.
Furthermore, when $S$ is a linear space, there exists a unique minimal face $F$ of $C_{d(x,S)}(x)$ that contains $C(x,S)$.
We use the sign vector $\sign(F)$ to give each $x \in \rr^m$ a combinatorial type as in the following definition.

\begin{defn}
	Let $L$ be a linear space and $F$ the minimal face of
	$C_{d(x,L)}(x)$ containing $C(x,L)$.
	The \emph{type of $x$ with respect to $L$} is $\type_L(x) := \sign(F)$.
\end{defn}

\begin{ex}
	Consider linear spaces $L_1 = \{(t,t) \in \rr^2: t \in \rr\}$ and $L_2 = \{(t,0) \in \rr^2 : t \in \rr\}$
	and let $x = (-3,-1)$ and $y = (5,3)$.
	Then $\type_{L_1}(x) = (+,-)$, $\type_{L_1}(y) = (-,+)$,
	$\type_{L_2}(x) = (0,+)$, and $\type_{L_2}(y) = (0,-)$.
	See Figure \ref{fig:typeexample} for an illustration.
	\begin{figure}[h!]
		\includegraphics[width=15cm]{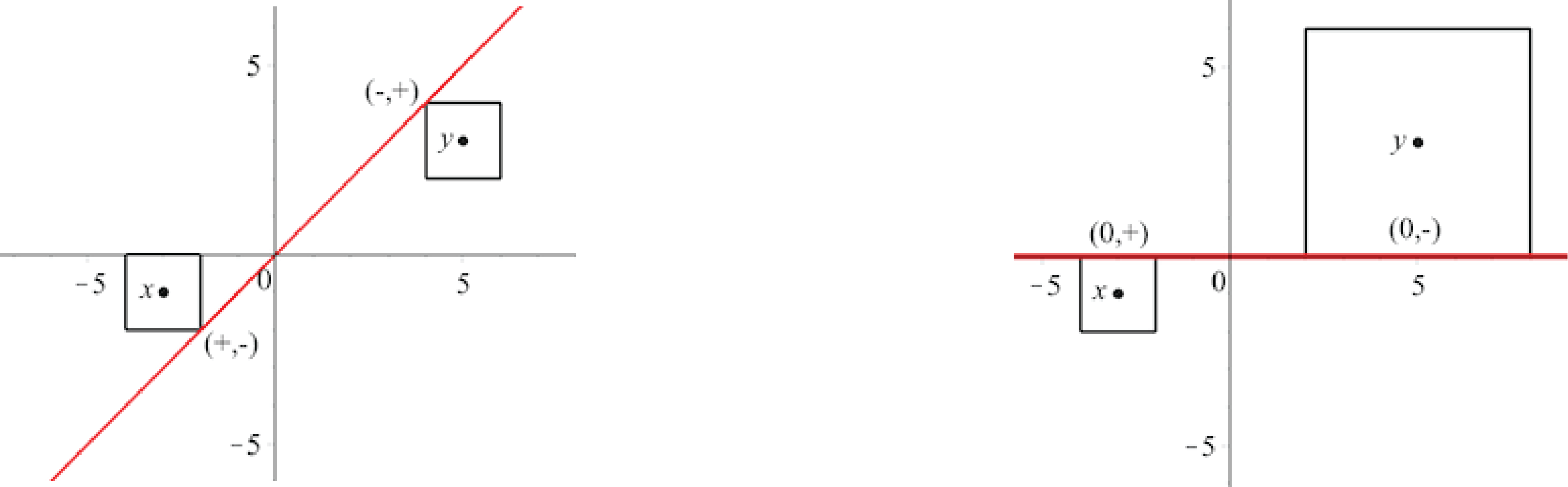}
		\caption{Types of $x$ and $y$ with respect to $L_1$ and $L_2$}\label{fig:typeexample}
	\end{figure}
\end{ex}

We will see that the sign vectors that can arise as the type of a point with respect to $L$
are precisely the elements of the oriented matroid associated to $L$. To aid in the proof we introduce the following convention for generating a vector with a given sign signature.

\begin{defn} 
	For $\sigma \in \{+,-,0\}^m$, $u(\sigma)$ is the vector in
	$\mathbb{R}^m$ with
 	\[
	u(\sigma)_i := 
		\left\{\begin{array}{rl}
			1 & \text{if}\quad\sigma_i = + \\
			-1 & \text{if}\quad\sigma_i = - \\
			0 & \text{if}\quad\sigma_i = 0 \\
 		\end{array}\right..
	\]
\end{defn}

\begin{lemma}\label{possibletypes}
	Let $L \subseteq \mathbb{R}^m$ be a linear space.
	Then the sign vectors that can arise as the type of a point with respect to $L$
	are precisely the elements of the oriented matroid associated to $L$.
	That is,
	\[
		\mathcal{O}_L = \{\type_L(x) : x \in \rr^m\}.
	\]
\end{lemma}

\begin{proof}  
	First, we will show that $\mathcal{O}_L \subseteq 
	\{\type_L(x) : x \in \rr^m\}$. 
	Let $\sigma \in \mathcal{O}_L$, we will show that the 
	type of $-u(\sigma)$ with respect to $L$ is equal 
	to $\sigma$. Since $\sigma \in \mathcal{O}_L$, by the definition of $\mathcal{O}_L$, there must exist a linear functional $c \in (\mathbb{R}^m)^*$ that 	vanishes on $L$ with $\sign(c) = \sigma$. 
	Now we claim that $d(-u(\sigma),L) = 1$.
	Since ${\bf{0}} \in L$
	and $d(-u(\sigma),{\bf{0}}) = 1$, $d(-u(\sigma),L) \le 1$.
	If $x \in \rr^m$ such that $d(-u(\sigma),x) < 1$,
	then for each index $i$ with $\sigma_i \neq 0$, $\sign(x_i) = -\sigma_i$.
	Therefore, $cx < 0$ which implies $x \notin L$. 
	Thus, $d(x,L) = 1$.
	\\
	\indent
	Next, we claim that $\type_L(-u(\sigma)) = \sigma$.
	Observe that
	$$F = \{y \in C_1(-u(\sigma)): y_i = 0 \text{ whenever } 
	\sigma_i \neq 0\}$$ 
	is a face of $C_1(-u(\sigma))$ and that $\sign(F) = \sigma$.
	Therefore, it will suffice to show that $F$ is the minimal face
	of $C_1(-u(\sigma))$ containing $C(-u(\sigma),L)$.
	So let $x \in C(-u(\sigma),L)$. We have already shown that 
	this implies that $x \in C_1(-u(\sigma))$. Moreover, for all $i$,  
	either $x_i= 0$ or 
	$\sign(x_i)	= \sign(-u(\sigma)_i) = -\sigma_i$. 	
	It must be the case then that if 
	$\sigma_i \not = 0$ then $x_i = 0$. Otherwise, $cx < 0$,
	which is impossible, since $c$ vanishes on $L$.
	Therefore, $x \in F$, and so $F$ contains $C(-u(\sigma),L)$.
	Finally, all that remains to show is that $F$ is the 	
	\emph{minimal} face of $C_1(-u(\sigma))$ containing 
	$C(-u(\sigma),L)$. If not, then there must exist  
	$j$ with $\sigma_j = 0$ such that $C(-u(\sigma),L)$ is 
	contained in a 
 	facet of $C_1(-u(\sigma))$ of the form 
	$\{y \in C_1(-u(\sigma)): y_j = 1\}$ or 
	$\{y \in C_1(-u(\sigma)): y_j = -1\}$. But this is impossible, 
	since we have already shown that 
	${\bf{0}} \in C(-u(\sigma),L)$. 
	Hence, 
 	$\type_L(-u(\sigma)) = \sigma$.
	\\
	\indent
	We now show $\{\type_L(x) : x \in \rr^m\} \subseteq \mathcal{O}_L$.
	Let $x \in \rr^m$.
	We will show that $\type_L(x) \in\mathcal{O}_L$.
	Assume $L$ has dimension $d$ and let $V \in \rr^{(m - d)
	\times m}$ be a matrix whose rows form a basis for 
	$L^\perp$.
	Let $\pi: \rr^m \rightarrow \rr^{m - d}$ be the map $x \mapsto 
	Vx$.
	Let $F$ be the minimal face of $C_{d(x,L)}(x)$ that contains 
	$C(x,L)$ so that $\type_L(x) = \sign(F)$.
	The image of $C_{d(x,L)}(x)$ under $\pi$ is the zonotope 
	associated to $V$ in $\mathbb{R}^{m - d}$. 
	Our goal will be to show
	that $F$ is the inverse image of one of the faces
	of this zonotope. By the definition of the oriented matroid 
	associated to a zonotope, this implies that $\sign(F)$ is an
	element of $\mathcal{O}_V$. 
	Proposition \ref{linearzonotopes} shows that
	the oriented matroids $\mathcal{O}_V$ and 
	 $\mathcal{O}_{\ker V}$ are equal.
	 Since $V$ is specifically 
	 constructed so that $\ker V = L$, this will also imply that 
	 $\type_L(x) = \sign(F) \in \mathcal{O}_L$.

	By the hyperplane separation theorem, there exists a 
	hyperplane separating $L$ and the interior of 
	$C_{d(x,L)}(x)$. Observe that any such hyperplane 
	must contain $L$ and intersect $F$ in its interior.
	Therefore, we may choose 
	$c \in (\rr^m)^*$ in the row-span of $V$ such that the 
	hyperplane $\mathcal{H}_c = \{y \in \rr^m : cy=0\}$
	is a face-defining hyperplane for $F$ and
	write $c = bV$ for some $b \in (\rr^{m - d})^*$.

	Since $\mathcal{H}_c$ is a face defining hyperplane of 	$C_{d(x,L)}(x)$,
	$\mathcal{H}_b = \{z \in \rr^{m - d} : bz=0\}$ must be a face-defining hyperplane 	of $\pi(C_{d(x,L)}(x))$. Clearly, $\pi(F)$ is contained in the face 
	of $\pi(C_{d(x,L)}(x))$ defined by $\mathcal{H}_b$.
	In fact, we have equality. If $z \in \pi(C_{d(x,L)}(x))$ then 
	$z= Vy$ for some $y \in C_{d(x,L)}(x)$. So if $bz = 0$, then
	$bz = (bV)y = cy = 0$. This implies that $y \in F$ and so
	 $z \in \pi(F)$. Similarly,
	if  $\pi(y) \in \pi(F)$ then 
	$0 = b(\pi(y)) = b(Vy) = (bV)(y) = cy$, which implies $y \in F$.
	Thus, we have just shown that $\pi(F)$ is a face of the 
	zonotope and that $\pi\inv(\pi(F)) = F$. Therefore, $\pi(F)$
	inherits its sign from $F$, and so 
	$\sign(F) \in \mathcal{O}_V$. As noted, this implies
	that $\type_L(x) = \sign(F) \in \mathcal{O}_L$.
\end{proof}

As we show in the following theorem,
the dimension of $C(x,L)$ depends entirely on $\type_L(x)$.
For any $\sigma \in \mathcal{O}_L$,
the \emph{rank of $\sigma$ in $\mathcal{O}_L$}, denoted $\rank(\sigma)$,
is the rank of the support of $\sigma$ in the matroid underlying $\mathcal{O}_L$.
So $\rank(\sigma)$ is the smallest number $k$ such that there exists indices $i_1,\dots,i_k$ with $\sigma_{i_j} \neq 0$
such that for all $y \in L$, if $y_{i_1} = \dots = y_{i_k} = 0$, then $y_j = 0$ for $\sigma_j \neq 0$.

\begin{thm}
\label{dimoftype}
	Let $L\subset \rr^m$ be a linear space of dimension $d$
	and let $\sigma \in \mathcal{O}_L$ be a sign vector in the oriented matroid associated to $L$.
	If $x \in \rr^m$ has $\type_L(x) = \sigma$,
	then the collection of $l^\infty$-closest points to $x$ in $L$ has dimension $d- \rank(\sigma)$.
%	That is, $\dim C(x,L) = d - \rank(\sigma)$
\end{thm}
\begin{proof}
	Let $L(\sigma)$ denote the linear space obtained by intersecting $L$
	and the $|\sigma|$ hyperplanes $\{x\in \rr^m: x_i = 0\}$ for $\sigma_i \neq 0$. 
	We claim that if $x \in \rr^m$ with $\type_L(x) = \sigma$, then $\dim C(x,L) = \dim L(\sigma)$.
	\\
	\indent
	Suppose $\type_L(x) = \sigma$, and let $F$ be the minimal
	face of $C_{d(x,L)}(x)$ containing $C(x,L)$.
	Let $y$ be a point in $C(x,L)$ that is also in the
	interior of $F$. Then given any point $z \in L(\sigma)$, it is 
	possible to choose $\varepsilon$ so that 
	$y + \varepsilon z \in F$ and hence in $C(x,L)$. 
	Therefore, $\dim C(x,L) \ge \dim L(\sigma)$. 
	Moreover, any two points in $C(x,L)$ are 
	contained in $F \cap L$ and so differ only by an 
	element of $L(\sigma)$. Therefore, 
 	$\dim C(x,L) \le \dim L(\sigma)$.
	\\
	\indent
	We now show that $\dim(L(\sigma)) = d-\rank(\sigma)$.
	Let $k := \rank(\sigma)$ and let $i_1,\dots,i_k$ be indices such that for all $y \in L$,
	$y_{i_1} = \dots = y_{i_k} = 0$ implies $y_j = 0$ when $\sigma_j \neq 0$.
	So $L(\sigma)$ can be expressed as the intersection of $L$ with the hyperplanes $\{x \in \rr^m : x_{i_j} = 0\}$,
	and by minimality of rank, this is not true of any subset of these hyperplanes.
	So $\dim(L(\sigma)) = d - \rank(\sigma)$.
	\begin{comment}
	\\
	\indent
	Now assume $\sigma$ is a circuit.
	Let $c_1,\dots,c_{m - d} \in (\rr^m)^*$ be linearly independent linear forms
	such that $L = \{x \in \rr^m : c_ix = 0, i=1,\dots,m - d\}$ with $\sign(c_1) = \sigma$.
	Then $L(\sigma)$ is defined by at most $(n - d) + |\sigma| - 1$ linearly independent forms, and so 
	$\dim(L(\sigma)) \ge d+1-|\sigma|$.
	Now if $\dim(L(\sigma)) > d+1-|\sigma|$,
	then there is some linear combination $a_2c_{2}+\dots +a_{m - d}c_{m - d}$
	whose set of non-zero entries is contained in that of $c_1$.
	So either $\{i: \sigma_i \neq 0\}$ is not a circuit of $\mathcal{M}_L$,
	or there is a dependence $b_1c_1 + b_2c_{2}+\dots +b_{m - d}c_{m - d} = 0$.
	Both are contradictions, so we must have $\dim L(\sigma) = d+1-|\sigma|$.
	\end{comment}
\end{proof}

\begin{ex}
	Let $L := \{(t,t,0) \in \rr^3 : t \in \rr\}$.
	Consider the points $x = (0,0,-1)$ and $y = (6,4,0)$.
	Then $\type_L(x) = (0,0,+)$ and $\type_L(y) = (-,+,0)$.
	Since $\rank(0,0,-) = 0$ and $d = 1$,
	Theorem \ref{dimoftype} tells us that $\dim(C(x,L)) = 1$.
	Since $\rank(+,-,0) = 1$,
	Theorem \ref{dimoftype} tells us that $\dim(C(y,L)) = 0$.
	Figure \ref{fig:typedim} shows $x$ and $y$ each surrounded by a cube of side length $2$ (colored red and light blue, respectively).
	The intersections with $L$ are $C(x,L)$ and $C(y,L)$.
	\begin{figure}[h!]
		\includegraphics[width=10cm]{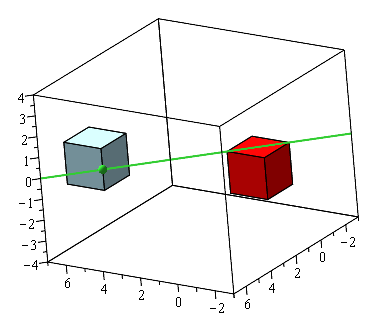}
		\caption{$L$ and cubes around $(0,0,-1)$ and $(6,4,0)$.}\label{fig:typedim}
	\end{figure}
\end{ex}

We can use the structure of the matroid $\mathcal{M}_L$ to glean information about possible values of $\dim(C(x,L))$.
Let $U_{d,m}$ denote the uniform matroid of rank $d$ on ground set $\{1,\dots,m\}$;
that is, the circuits of $U_{d,m}$ are all $d+1$-element subsets of $\{1,\dots,m\}$.

\linearunique
\begin{proof}
	Let $d$ be the dimension of $L$.
	If $\mathcal{M}_L$ is \emph{not} uniform,
	then $\mathcal{O}_L$ has a circuit $\sigma$ with $|\sigma| \le d$, so $\rank(\sigma) \le d-1$.
	Then Lemma \ref{possibletypes} and Theorem \ref{dimoftype} imply that there exists a point
	 $x \in \rr^m$ such that $\dim C(x,L) = d-\rank(\sigma)\ge 1$.
	 \\
	 \indent
	 If $\mathcal{M}_L = U_{d,m}$ then $\rank(\sigma) = d$ for all $\sigma \in \mathcal{O}_L$.
	 Theorem \ref{dimoftype} implies that $\dim(C(x,L)) = 0$ for all $x \in \rr^m$.
\end{proof}

Lemma \ref{possibletypes} enables us to give a partition of $\mathbb{R}^m$ by type with respect to $L$. 

\begin{prop} \label{type space}
	Let $L \subseteq \rr^m$ be a linear space and let $\sigma \in \mathcal{O}_L$ be a sign vector in the oriented matroid associated to $L$.	
	The set of all points in $\mathbb{R}^m$
	 with type $\sigma$ with respect to $L$ is the Minkowski
	 sum of $L$ and the interior of the conical hull of 
	 $\{ -u(\tau) : \sigma \preceq \tau \}$. That is,
	 
	 $$
	 \{x \in \mathbb{R}^m : \type_L(x) = \sigma\} = 
	 L + \textnormal{int}(\cone(\{ -u(\tau) : \sigma \preceq \tau \})).	$$
\end{prop}

\begin{proof}
	Let $\sigma \in \mathcal{O}_L$.
	Define $\mathcal{V}_\sigma := \{ -u(\tau) : \sigma \preceq \tau \}$. 
	First, we will show that everything in 
	$L +\textnormal{int}(\cone(\mathcal{V}_\sigma))$ has type $\sigma$. 
	Since adding an element of $L$ 
	 to a point does not change its type with respect to
	 $L$, it will suffice to show everything in 
	 $\textnormal{int}(\cone(\mathcal{V}_\sigma))$ has type $\sigma$. 
	 \\
	 \indent
	 Let $x \in \textnormal{int}(\cone(\mathcal{V}_\sigma))$. 
	 Then there exists 
	$\alpha > 0 $ such that if 
	$\sigma_i = +$ or $\sigma_i = -$,
	 then $|x_i| = \alpha$  and
	$|x_i| < \alpha$ otherwise.
	By Lemma \ref{possibletypes}, there exists 
	$c \in (\mathbb{R}^m)^*$ such that $\sign(c) = \sigma$
	and $cy = 0$ for all $y \in L$.
	Let $\mathcal{H}_c := \{y \in \rr^m : cy = 0\}$ be the hyperplane defined by $c$.
	It is clear that $d(x,\mathcal{H}_c) = \alpha$, and
	any $y \in C(x,\mathcal{H}_c)$ must have $y_i = 0$ if 
	$|x_i| = \alpha$. 
	Since $L \subseteq \mathcal{H}_c$, the same is true for
	each $y \in C(x,L)$. Therefore, if 
	$\sigma_i = +$ or $\sigma_i = -$, then
	$\sigma_i =\type_L(x)_i$.
	\\
	\indent
	Sine $d(x,\mathcal{H}_c) = \alpha$ and $L \subseteq \mathcal{H}_c$, 
	$d(x,L) \geq \alpha$.
	Since $d(x,{\bf{0}}) = \alpha$ and ${\bf{0}} \in L$, this implies $d(x,L) = \alpha$ and ${\bf{0}} \in C(x,L)$.
	 If $\sigma_i = 0$ then $|x_i - 0| = |x_i| < \alpha$, which
	 implies that if $\sigma_i = 0$, $\type_L(x)_i = 0$. Thus, 
	 $\type_L(x) = \sigma$. 
	 \\
	 \indent
	 To see that everything of type $\sigma$ is contained in $L + \textnormal{int}(\cone(\mathcal{V}_\sigma))$,
	 let $x$ be such that $\type_L(x) = \sigma$.
	 By definition of type, this means that $\sigma$ is the unique sign vector such that
	 if $F$ is the face of the unit cube $C_{1}({\bf 0})$ with type $\sigma$,
	 then there exists some $y \in \textnormal{int}(F)$ such that $x + \lambda y \in L$ for some $\lambda > 0$.
	 So there exists some $l \in L$ such that $x = l+\lambda(-y)$
	 thus showing that $x \in L + \textnormal{int}(\cone(\mathcal{V}_\sigma))$.
\end{proof}

Modulo the lineality space $L$, the closures of the cones in Proposition \ref{type space}
form the face fan of the zonotope obtained by projecting the cube $C_{d(x,L)}(x)$ onto $L^\perp$.

\begin{cor}
	Let $x \in \rr^m$ and let $V$ be a matrix whose rows span $L^\perp$.
	Then $\type_L(x)$ is
	equal to the sign of the unique face $F$ of $Z(V)$
	such that the conic hull of the interior of $F$ contains $Vx$.
\end{cor}

\indent
The signs of the facets of $Z(V)$ correspond to circuits of $\mathcal{O}_L$ \cite[Corollary 7.17]{ziegler2000}.
This implies that the full dimensional cones of the partition correspond to circuits.
Hence $\type_L(x)$ is generically a circuit of $\mathcal{O}_L$.

\section{Applications to Phylogenetics}
\label{sec:phylogenetics}

In this section, we will consider how the results above can 
be applied to phylogenetic reconstruction using the $l^\infty$-metric. 
We will address Problem \ref{pr:phylogenetics}, concerning the structure of the set of $l^\infty$-closest points to 
the set of equidistant tree metrics.
In particular, we show that there can be many (equidistant) tree metrics that are equally close to a given dissimilarity map,
and they can represent many different tree topologies.
We also decompose the space of dissimilarity maps on 3
elements and on 4 elements according to the tree topologies represented in the 
set of $l^\infty$-closest equidistant tree metrics.
Finally, we investigate optimizing to the set of tree metrics
and show how many of the results for 
equidistant tree metrics carry over.

\subsection{Rooted trees and ultrametrics}

Let $RP(n)$ be the set of all $n$-leaf rooted trees with leaves labeled by $[n] := \{1,\ldots,n\}$. 
Following the convention of 
 \cite[Section 2.2]{Semple2003}, we call the elements of $RP(n)$
\emph{rooted phylogenetic $[n]$-trees}. We will also consider
the set of \emph{rooted binary phylogenetic $[n]$-trees}
which we will denote $RB(n)$.
A \emph{polytomy} of a non-binary tree is a vertex
with degree greater than three - that is, a witness to the property of being non-binary.
To represent the topology of $\mathcal{T} \in RP(n)$ we use 
the notation $(S_1(S_2))$ to indicate that the leaves labeled by
the set $S_1$ and $S_2$ are on opposite sides of the root in $\mathcal{T}$. 
We apply this notation recursively to give the topology of the rooted subtree in $\mathcal{T}$ 
induced by the labels in $S_1$ and $S_2$.
Thus, for example, we can express the topology of the
the rooted tree in Figure \ref{fig:ultrametrictrees} by $(D(C(AB))).$

Let  $\mathcal{T} \in RP(n)$ and assign a positive weighting to the edges of $\mathcal{T}$. This naturally induces a metric $\delta$ on the leaves of $\mathcal{T}$ where $\delta(i ,j)$ is the
sum of the edge weights on the unique path between $i$ and $j$. 
If we further assume that the distance from each leaf vertex to the root is the same then $\delta$ is an \emph{ultrametric}.

\begin{defn}\cite[Definition 7.2.1]{Semple2003} 
A dissimilarity map $\delta : X \times X \rightarrow \mathbb{R}$ is called an \emph{ultrametric on $X$} if for every three distinct elements $i,j,k \in X,$
$$\delta(i,j) \leq \max\{\delta(i,k), \delta(j,k) \}.$$
\end{defn}

An equidistant edge weighting of a rooted tree is a weighting of the edges where the distance from each leaf to the root is the same and where the weight of every internal edge is positive. 
Note that this allows the possibility of nonpositive weights on leaf edges. Given any ultrametric $u$ on $[n]$, there exists a unique 
$\mathcal{T} \in RP(n)$ 
and an equidistant weighting $w$ such that the ultrametric induced by $(\mathcal{T}:w)$ is equal to $u$ 
\cite[Theorem 7.2.8]{Semple2003}. 
We call $(\mathcal{T}:w)$ an 
\emph{equidistant representation of $u$} and say that
$\mathcal{T}(u) := \mathcal{T}$ is the
\emph{topology of $u$}. 

We can also convert an equidistant edge
weighting of a tree into a \emph{vertex weighting} 
of that same tree \cite[Theorem 7.2.8]{Semple2003}. 
Given any internal vertex $v$ in an equidistant representation of an ultrametric $u$, $u(i,j)$ is constant over all pairs of leaves $i,j$ having $v$ as their most recent common ancestor.
We obtain a vertex weighting from an edge weighting by labeling each internal vertex by this constant value.
For what follows, we will represent dissimilarity maps on 
$n$ elements as points in $\mathbb{R}^{n \choose 2}$ by 
letting $\delta_{ij} = \delta(i,j)$ and use 
$U_n  \subseteq \mathbb{R}^{n \choose 2}$
 to denote the set of all ultrametrics on $n$ elements. 
Many of our examples will involve dissimilarity maps on $4$ 
elements, in which case we let 
$(x_{AB},x_{AC},x_{AD},x_{BC},x_{BD},x_{CD})$
be the coordinates of an arbitrary point
in $\mathbb{R}^{4 \choose 2}$. We will also use
the notation $e_{ij}$ to denote the standard basis
vector with $x_{ij} = 1$ and all other entries equal to
zero.
\\
\begin{ex} 
Consider the ultrametric 
$u = (5,7,9,7,9,9) \in \mathbb{R}^6$.
Figure \ref{fig:ultrametrictrees} shows two equivalent ways 
of representing $u$: with a vertex weighting 
on the left and an equidistant edge weighting on the right.

\begin{figure}[h]
\includegraphics[width=14cm]{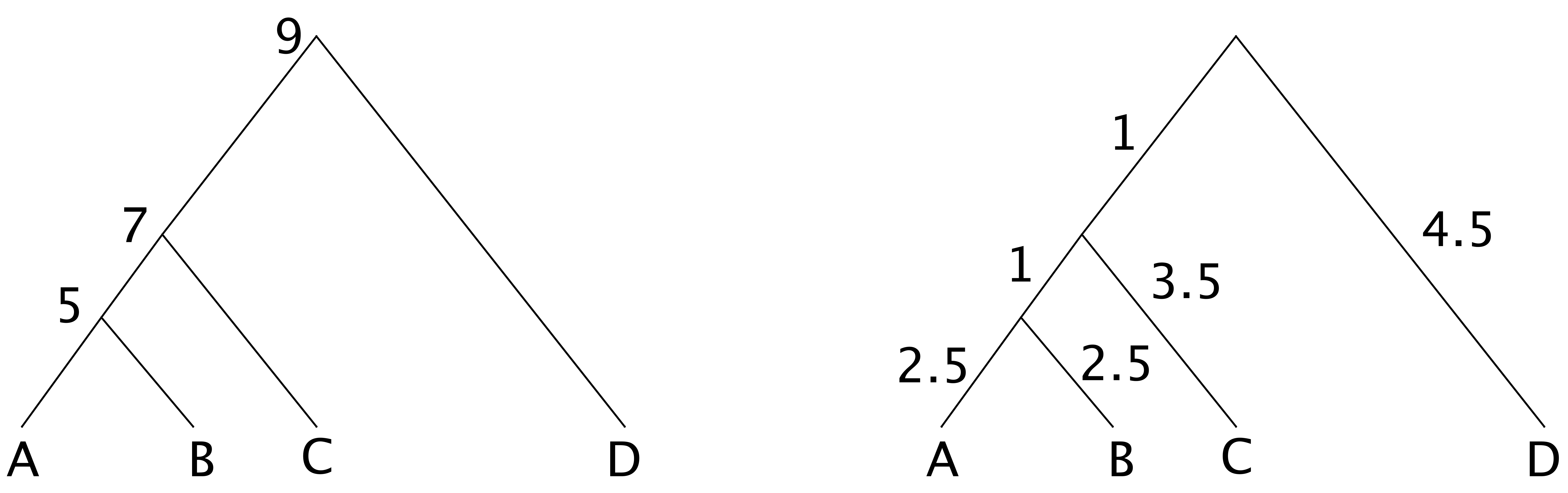}
\caption{Two different representations of $u = (5,7,9,7,9,9)$.}
\label{fig:ultrametrictrees}
\end{figure}

\end{ex}

%\begin{proof}
%	Ardila and Klivans have shown \cite{ardila-klivans2006} that two $\calm(K_n)$-ultrametrics $u,w$ have the same rooted tree topologies
%	if and only if $\calm(K_n)_u = \calm(K_n)_w$.
%\end{proof}

%%There are several methods for doing so \cite[Chapter 7]{Semple2003}. 
%An obvious choice is to select the tree metric that minimizes the euclidean distance to the dissimilarity map, this is the problem of \emph{least squares phylogeny} (LSP). Because of the computational complexity of LSP, 
%polynomial time heuristics such as
%neighbor joining (NJ) and unweighted pair
%group method with arithmetic mean (UPGMA)
%have been developed to approximate the 
%least-squares phylogeny.
%The latter method actually returns an
%equidistant tree metric. 

\subsection{$l^\infty$-optimization to the set of Ultrametrics}

Given a dissimilarity map $\delta$, we let $\delta_U$ be the unique coordinate-wise maximum ultrametric 
which is coordinate-wise less than $\delta$. 
This is called the 
\emph{subdominant ultrametric of $\delta$}. 
The existence and uniqueness of the subdominant
ultrametric are proven and a polynomial time algorithm 
for computing it is given in \cite[Chapter 7]{Semple2003}.

Our interest in the subdominant ultrametric is that it gives us a 
way to determine an $l^\infty$-closest ultrametric
to a dissimilarity map $\delta$ \cite{chepoi}.
We first compute the subdominant ultrametric 
$\delta_U$ and then use that 
$$d(\delta, U_n) = \frac{1}{2}d(\delta, \delta _U) .$$
We define $\delta_c$, the 
\emph{canonical closest ultrametric to $\delta$}, by
$\delta _c(i,j) = \delta_U(i,j)  + \frac{1}{2}d(\delta, \delta _U)$.

As noted in the introduction, the set of $l^\infty$-closest ultrametrics is in general not a single point.
Moreover, in many instances, the set of $l^\infty$-closest ultrametrics to $\delta$, 
$C(\delta , U_n)$, will contain ultrametrics representing different topologies.
Thus, there may be several different trees that explain the data equally well from the perspective of the $l^\infty$-metric.

\begin{ex}
\label{ex:closestultrametrics}
	Figure \ref{fig:closestultrametrics} depicts three ultrametrics in 	
	the set of closest ultrametrics to 
	$\delta = (2,4,6,8,10,12)$.
        The subdominant ultrametric is
        $\delta_U = (2,4,6,4,6,6)$, $d(\delta,U_4) = 3$, and the canonical closest 	
        ultrametric
        (pictured far left) is $\delta_c = (5,7,9,7,9,9)$.

%	So this is an example of a dissimilarity map who has $l^\infty$-ultrametrics of differing tree topologies. 
	\begin{figure}[h]
		\includegraphics[width=13cm]{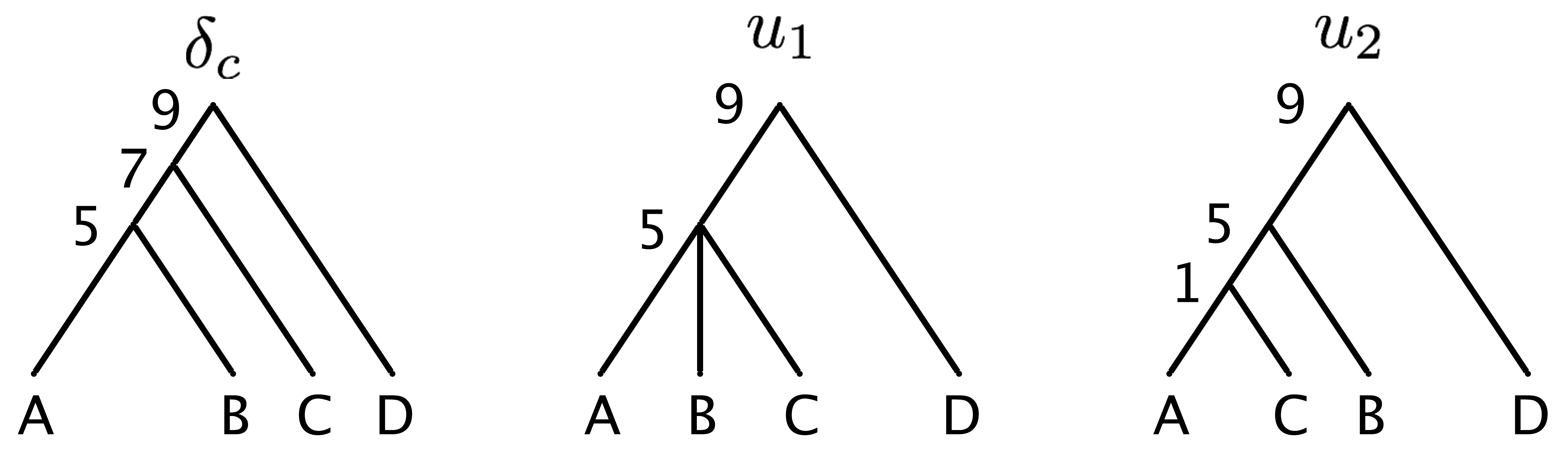}
		\caption{Three ultrametrics in $C(\delta, U_4)$ for 
		$\delta = (2,4,6,8,10,12)$.}
		\label{fig:closestultrametrics}
	\end{figure}
	
	One can easily verify that the canonical closest ultrametric 
	inherits dominance from the subdominant ultrametric. That     is, for any $\delta \in \mathbb{R}^{n \choose 2}$, every ultrametric in $C(\delta, U_n)$ is coordinate-wise less than $\delta_c$. Thus, we can construct closest ultrametrics by ``sliding down" vertices of $\delta_c$ so long as
	the $l^\infty$-distance between $\delta$ and the new 	
	ultrametric does not exceed $d(\delta,\delta_c)$.
	In Figure \ref{fig:closestultrametrics},
	we obtain $u_1$ from $\delta_c$ by sliding the middle 	
	internal vertex until it reaches the lowest one.
	We obtain $u_2$ by continuing to slide this vertex until we 
	can do so no more. Observe that in this case, the root vertex 	must remain fixed.
	
\end{ex}

Example \ref{ex:closestultrametrics} shows that it is possible for the 
set of $l^\infty$-closest ultrametrics to a point to contain different topologies. 
Below, we consider what sets of topologies are represented
in $C(\delta,U_n)$ for an arbitrary point
 $\delta \in  \mathbb{R}^{n \choose 2}$.
The idea behind most of these proofs is to find a linear space that contains the ultrametrics for many different tree
topologies and apply the constructions for linear spaces
developed in Section \ref{sec:linearspaces}.

\begin{defn} 
Let $\delta \in \mathbb{R}^{n \choose 2}$ and $u \in U_n$.
Define
$$Top(\delta) := \{ \mathcal{T}(u) : \  u \in C(\delta ,U_n)  \}.$$ 
\end{defn}

In \cite{davidson-sullivant2014}, the authors study the geometry of the set of dissimilarity maps around a polytomy with respect to the Euclidean norm. They showed that locally this space could be partitioned according to the closest tree topology. 
Proposition \ref{prop:polytomy} shows the contrast between that 
situation and using the $l^\infty$-metric. 

\begin{prop} 
\label{prop:polytomy}
Let $\mathcal{T}$ be a rooted phylogenetic $[n]$-tree with a polytomy.
Assume that $\calt$ is not the star tree.
Then there exists $\delta \in \mathbb{R}^{n \choose 2}$ such
that $Top(\delta)$ contains $\mathcal{T}$ and all of its resolutions.
\end{prop}
\begin{proof} 
Let $u$ be an ultrametric with $\mathcal{T}(u) = \mathcal{T}$.
Since $\calt$ is not the star tree, there exist three leaves $\{i,j,k\}$ such that $u_{ij} < u_{ik} = u_{jk}$.
Define $\delta := u + \varepsilon(e_{ik} - e_{jk})$ for some $0 < \varepsilon < u_{ik}-u_{ij}$. 
Then $u$ is in $C(\delta ,U_n)$ and so are all possible resolutions of the polytomy. 
\end{proof}

\begin{ex} Let $(x_{AB},x_{AC},x_{AD},x_{BC},x_{BD},x_{CD})$ be the coordinates of a point in $\mathbb{R}^{4 \choose 2}$ and consider $u = (5,5,10,5,10,10) \in \mathbb{R}^{4 \choose 2}$. The topology of the ultrametric $u$ is the rooted tree $(D(ABC))$ with an unresolved tritomy.  

Note that $u_{BC} <  u_{BD} = u_{CD}$. Choose
$\varepsilon = 1$ and let 
$$\delta = u + \varepsilon(e_{CD} - e_{BD}) =
(5,5,10,5,9,11).$$ 
The canonical closest ultrametric $\delta_c = (6,6,10,6,10,10)$ also has an unresolved tritomy and $C(\delta ,U_4)$ contains ultrametrics corresponding to each different resolution.
For example, $(4,6,10,6,10,10), 
(6,4,10,6,10,10),$ and
$(6,6,10,4,10,10)$ are all elements of $C(\delta ,U_4)$. 
\end{ex}

We obtain the following corollary by choosing a tree with a single resolved triple in the proof of Proposition \ref{prop:polytomy}.

\begin{cor} 
\label{cor:tritomy}
There exist points in $\mathbb{R}^{n \choose 2}$ for which $Top(\delta) \cap RB(n)$ contains $(2n - 3)!!/3$ different tree topologies.
\end{cor}

We will also see from our decomposition of
$\mathbb{R}^{4 \choose 2}$ 
that there are actually 6-dimensional polyhedral cones 
in which every point in the interior has 
five $l^\infty$-closest binary tree topologies.
\\
\indent
Even when all $l^\infty$-closest ultrametrics to some given $\delta \in \rr^{\binom{n}{2}}$ have the same topology,
the dimension of the set of $l^\infty$-closest ultrametrics can be high.
The affine hull of each maximal cone of $U_n$ is a linear space defined by relations of the form $x_{ik} - x_{jk} = 0$ where $(k(ij))$ is a triple compatible with the corresponding tree.
As before, we can find points where optimizing to $U_n$ is equivalent to optimizing to such a linear space and so our results from Section \ref{sec:linearspaces} can be applied.

\begin{prop}
\label{prop:dim}
Let $\mathcal{T} \in RB(n)$. There exists $\delta \in \mathbb{R}^{n \choose 2}$ such
that $\dim(C(\delta ,U_n)) = n - 2$ and every ultrametric in $C(\delta ,U_n)$ has topology $\mathcal{T}$.
\end{prop}
\begin{proof} Let $u$ be an ultrametric and $(k(ij))$ a triple compatible with $\mathcal{T}(u)$. 
For $\varepsilon > 0$, let $\delta = u + \varepsilon(e_{ik} - e_{jk})$. If $\varepsilon$ is sufficiently small, $C(\delta ,U_n) = C(\delta ,L)$ where $L$ is the affine hull of the maximal cone of $U_n$ containing $u$. The type of $x$ relative to $L$ is the sign vector $\sigma$ where $\sigma_{ij} = +$, $\sigma_{ik} = -$, and all other entries are zero.
The rank of $\sigma$ in $\mathcal{O}_L$ is one, and
thus by Theorem \ref{dimoftype}, $\dim(C(\delta ,U_n)) = (n-1) - 1 =  n - 2$.
\end{proof}

\begin{ex} 
\label{dim of C}
Let $(x_{AB},x_{AC},x_{AD},x_{BC},x_{BD},x_{CD})$ be the coordinates of a point in $\mathbb{R}^{4 \choose 2}$.
Choose $u = (5,7,9,7,9,9)$, the ultrametric corresponding to the tree at far left in Figure \ref{fig:ultrametrictrees}. 
We will perturb $u$ to construct a dissimilarity map 
$\delta$ where the set of $l^\infty$-closest
points to $\delta$ has dimension two.
The triple $(C(AB))$ is compatible with 
$\mathcal{T}(u)$ and so we let  
$$\delta =
(5,7,9,7,9,9) + (e_{AC} - e_{BC})= 
(5,8,9,6,9,9).$$
The subdominant ultrametric $\delta_U = (5,6,9,6,9,9)$ and the canonical ultrametric $\delta_c=(6,7,10,7,10,10)$. 
We have two degrees of freedom that come from adjusting the values of $\{(\delta_c)_{AD},(\delta_c)_{BD},(\delta_c)_{CD}\}$ (sliding down 
the root) or $\{(\delta_c)_{AB}\}$ (sliding down the most recent common ancestor of $A$ and $B$). 
\end{ex}

\subsection{The Decomposition for 3-Leaf and 4-Leaf Trees}

The following definition makes formal the idea of partitioning 
the points in $\mathbb{R}^{n \choose 2}$ according to their
sets of $l^\infty$-closest trees.

\begin{defn}
Let 
 $\{\mathcal{T}_1, \ldots, \mathcal{T}_k\}
  \subseteq 
 RP(n)$.
The \emph{district} of 
$\{\mathcal{T}_1, \ldots, \mathcal{T}_k\}$ is the set
 $$D(\{\mathcal{T}_1, \ldots, \mathcal{T}_k\}) := 
 \{\delta \in \mathbb{R}^{n \choose 2} :
 Top(\delta) 
  = \{\mathcal{T}_1, \ldots, \mathcal{T}_k\} \}.
 $$
\end{defn}

We can represent a dissimilarity map on three elements as a point $(x_{12},x_{13},x_{23}) \in \mathbb{R}^3$. 
There are three maximal cones of $U_3$ corresponding to the 
three elements of $RB(3)$. 
Modulo the common lineality space of each of these cones, 
$\text{span}\{(1,1,1)\}$,
we can fix the first coordinate at zero and represent the space
of dissimilarity maps on three elements in the plane. 

Figure \ref{fig:3leaftreespace} depicts
a polyhedral subdivision of $\mathbb{R}^3$ according to districts. There are seven cones in this subdivision. The labels
$(1(23)), (3(12))$, and $(2(13))$ label the image of the set of ultrametrics for each topology. These labels also label the areas between the dashed lines which are the 2-dimensional images of the three 3-dimensional districts
$D\{(1(23))\}$,
$D\{(2(13))\}$, and
$D\{(3(12))\}$.
The dotted lines themselves are the images of the 2-dimensional cones whose interiors form the districts
$D\{(123), (1(23)), (2(13))\}$,
$D\{(123),(1(23)), (3(12))\}$, and
$D\{(123),(2(13)), (3(12))\}$.
The origin represents the image of $\text{span}\{(1,1,1)\}$ which is the district of the 3-leaf claw tree, $D\{(123)\}$.

\begin{figure}
		\includegraphics[width=8cm]{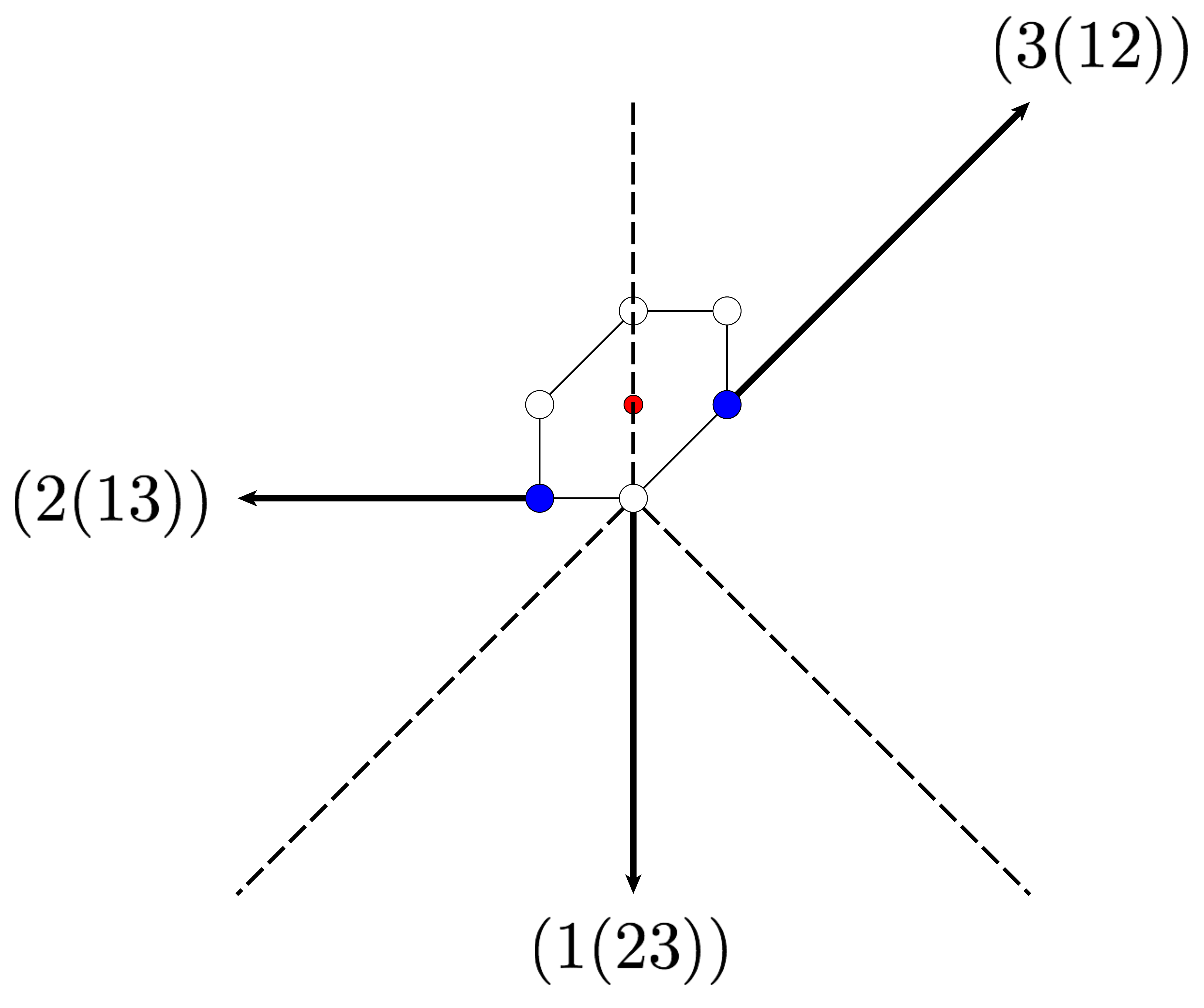}
		\caption{A 2-dimensional representation of the polyhedral subdivision of $\mathbb{R}^3$ according to district.}\label{fig:3leaftreespace}
	\end{figure}
	
\begin{ex}
\label{ex:zonotope}
	The image of the dissimilarity map $\delta = (1,1,3)$
	after modding out by $U_3$'s lineality space
	is pictured in Figure \ref{fig:3leaftreespace}.
	Note that $d(\delta,U_3) = 1$.
	The hexagon surrounding it is the zonotope that is the image of the cube $C_{1}(\delta)$.
	The filled vertices of the zonotope
	correspond to the fully resolved 
	$l^\infty$-closest ultrametrics $(2,1,2)$ and $(1,2,2)$.
	The origin corresponds to the $l^\infty$-closest ultrametric 
	 $(2,2,2)$ and 
	$\delta \in D\{(123),(2(13)), (3(12))\}$.
	
\end{ex}

The decomposition for $4$-leaf trees is much more complicated. The supplemental materials, located at 
\begin{center}
\bigskip
\url{http://www4.ncsu.edu/~dibernst/Supplementary_materials/L-infinity.html}
\bigskip
\end{center}

\noindent
contain a Maple \cite{Maple} file for computing a 
polyhedral subdivision of $\mathbb{R}^{4 \choose 2}$ into a fan consisting of 723 maximal polyhedral cones labeling 37 different districts. Each maximal cone is labeled by a set 
$\{\mathcal{T}_1, \ldots, \mathcal{T}_k\} \subseteq RP(n)$, meaning that each point in the interior of the cone is in
 $D(\{\mathcal{T}_1, \ldots, \mathcal{T}_k\} )$.

The computations rely heavily on the functionality of the package {\tt{PolyhedralSets}} (available in Maple2015 and later versions).
The fan is computed by first considering each of the fifteen different trees in $RB(4)$ individually.  
For each $\mathcal{T} \in RB(4)$, we construct a fan with support $\mathbb{R}^{4 \choose 2}$, where all of the points in the interior of each maximal cone in
 the fan satisfy either
 $\mathcal{T} \in Top(\delta)$ or 
 $\mathcal{T} \not \in Top(\delta)$. 
The resulting polyhedral subdivision is the common refinement of these fifteen fans. 
Note that there are far more than 37 districts since
our construction only labels the 6-dimensional districts.
 
%
%\begin{ex}
%Lots of closest trees.
%\end{ex}

Based on the 3-leaf case, one might hope that districts are
easily described or possess some nice properties. For example, the 3-leaf districts are all convex and tropically convex.
However, the 4-leaf case shows that many of these properties do not hold in 
general. For the rest of this section, 
we let 
$(x_{12},x_{13},x_{14},x_{23},x_{24},x_{34})$ denote an arbitrary point in
$\mathbb{R}^{4 \choose 2}.$ 
We include some results
for those familiar with tropical geometry 
and its connections to
phylogenetics without the requisite background that
would take us too far afield.

\begin{prop} Districts are not necessarily convex nor tropically convex.
\end{prop}

\begin{proof}
\label{not convex}
We offer the following counterexample in $\mathbb{R}^{4 \choose 2}.$
Let $\delta^1 = (10,20,21,23,25,27)$ and 
$\delta^2 = (10,23,21,20,25,27)$. Not only is 
$Top(\delta^1) = Top(\delta^2) =  \{ (4(3(12))) \}$, but
in fact  
$\delta^1_U = \delta^2_U = (10,20,21,20,21,21)$.
The point
$\delta^3 = \frac{1}{2}\delta^1 + \frac{1}{2}\delta^2 $ lies on 
the line between these two points 
but $Top(\delta^3) =  \{ (3(4(12)))\}$.

 Similarly, using the operations of the max-plus algebra, the point
$\delta^4 =  (0\odot \delta^1) \oplus (- \frac{3}{2}\odot \delta^2 )$ 
lies on the tropical line between these two points but 
$Top(\delta^4) =  \{ (3(4(12)))\}$.
\end{proof}

Let $\delta$ and $\delta'$ be dissimilarity maps in $\mathbb{R}^{n \choose 2}$. From the algorithm for computing the subdominant ultrametric, it is clear that $\delta$ and $\delta'$ will have the same subdominant ultrametric topology if they have the same relative ordering of coordinates - that is, $\delta_{ij} \leq \delta_{kl}$
if and only if $\delta'_{ij} \leq \delta'_{kl}$ and
$\delta_{ij} < \delta_{kl}$
if and only if $\delta'_{ij} < \delta'_{kl}$.
For 3-leaf trees, relative ordering also completely
determines district. The example below 
demonstrates that for trees with more
 than three leaves this is not the case.

%One might try to decompose $\rr^{\binom{n}{2}}$ by district
%by using the decomposition of space induced by a tropical polyhedral cone from \cite{develin-sturmfels2004}
%along with the tropical V-description of $U_n$ given in \cite{ardila2004}.
%In this decomposition, if $\delta^1,\delta^2 \in \rr^{\binom{n}{2}}$ are in the same cell,
%then their subdominant ultrametrics $\delta^1_u,\delta^2_u$,
%and therefore their canonical closest trees $\delta^1_U, \delta^2_U$, have the same topology.
%However, the converse is not true.
%Moreover, even when $\delta^1_U$ and $\delta^2_U$ have the same topology,
%$\delta^1$ and $\delta^2$ may not be in the same district..
%\\
%\indent
%Two arbitrary points $\delta^1$ and $\delta^2$ have the same
%subdominant ultrametric topology if they have the same relative ordering of coordinates - that is, $\delta^1_{ij} \leq \delta^1_{kl}$
%if and only if $\delta^2_{ij} \leq \delta^2_{kl}$. 
%However, as the following example shows,

%Two arbitrary points
% $\delta^1$ and $\delta^2$ have the same
%subdominant ultrametric topology if they have the same relative ordering of coordinates - that is, $\delta^1_{ij} \leq \delta^1_{kl}$
%if and only if $\delta^2_{ij} \leq \delta^2_{kl}$. 
%Thus, to deal with the example above we might consider further partitioning the set of points according to the relative order of their coordinates. The following example shows that this partition is still not fine enough to account for districts.

\begin{ex}
Let $\delta^1 = (4,8,12,9,21,22)$ and $\delta^2 = (4,8,12,9,13,14)$.
Both dissimilarity maps satisfy 
$\delta_{12}^i < \delta_{13}^i < \delta_{23}^i < \delta_{14}^i < \delta_{24}^i < \delta_{34}^i$.
However, $Top(\delta^1) = \{ (4(3(12))) \}$ and
$Top(\delta^2) =  \{ (4(3(12))), (4(2(13))), (4(1(23))) \}$.
\end{ex}

It does not appear possible to simplify the given subdivision 
of $\mathbb{R}^{4 \choose 2}$ much further by combining cones. Consider for example the forty maximal cones that constitute the district $D(\{(4(3(12)))\})$. 
Any five element subset of these cones contains a pair whose 
convex hull has full dimensional intersection with the interior of a
 maximal cone from another district.
Therefore, by combining these cones the best we could hope
for is to represent this district as the union of ten maximal convex cones. While a few can be patched together the final description does not appear any simpler. 

This polyhedral subdivision was constructed by 
examining each possible 4-leaf subdominant ultrametric topology and writing out inequalities to determine when we could  obtain a new topology. 
It is certainly possible, though likely much more
difficult, to do the same thing for trees with any fixed number of leaves. It is unclear how to generalize our approach to an arbitrary number of leaves and so the following problem remains open.

\begin{pr}
	Give a polyhedral decomposition of $\rr^{\binom{n}{2}}$ according to districts.
%	That is, find a polyhedral decomposition such that if $\delta^1,\delta^2$ are in the interior of the same cell,
%	then $\delta^1$ has an $l^\infty$-closest ultrametric with tree topology $\calt$ if and only if $\delta^2$ does too.
\end{pr}

\subsection{Tree Metrics} 

We end with a note about $l^\infty$-optimization to the set of tree metrics. A tree metric $\delta$ on $[n]$ is a metric induced by a positive edge weighting of an $n$-leaf tree (no longer rooted nor equidistant). The pair $(\mathcal{T}:w)$ that realizes this metric is called a \emph{tree metric representation of $\delta$}. A metric $\delta$ is a tree metric if and only if it satisfies the four-point condition \cite[Theorem 7.2.6]{Semple2003} . 

\begin{defn}
\label{4 point}
\cite[Definition 7.2.1]{Semple2003} 
A dissimilarity map $\delta : X \times X \rightarrow \mathbb{R}$ satisfies the four-point condition if for every four (not necessarily distinct) elements $w,x,y,z \in X$, 
$$\delta(w,x) + \delta(y,z) \leq
 \max \{ \delta(w,y) + \delta(x,z), \delta(w,z) + \delta(x,y)\}.$$
\end{defn}

We use the notation $\mathcal{T}_n \subset \mathbb{R}^{n \choose 2}$ to denote the set of all 
tree metrics on $[n]$. If we insist that the points in Definition \ref{4 point} are distinct, then the set of metrics satisfying the \emph{distinct} 4-point condition is the tropical Grassmannian \cite{Speyer}.
 Thus, the problem
of finding the closest tree metric is closely related to the problem of $l^\infty$-optimization to this tropical variety. 

Although there is no subdominant tree metric, we can still compute the $l^\infty$-distance from an arbitrary point to the set of tree metrics. The set of binary phylogenetic trees with label set $[n]$ is $B(n)$. For each $\mathcal{T} \in B(n)$, the distance to the set of tree metrics with topology $\mathcal{T}$ can be found by solving a linear program. 
Taking the minimum of these $(2n-5)!!$ individually 
computed distances
gives us the distance to the set of tree metrics.
Results in a forthcoming paper show that $C(\delta,U_n)$ is always a tropical polytope and thus connected. 
This is a nice property from the perspective of phylogenetic
 reconstruction, as it means we have a set of 
closest ultrametrics any of which can be obtained from
another by shrinking and growing branch lengths without ever leaving the set $C(\delta,U_n)$. The same does not hold for tree metrics.

\begin{prop} 
\label{prop: not connected}
There exists 
$\delta \in \mathbb{R}^{6 \choose 2}$ 
such that $C(\delta,{\mathcal{T}_6})$ and 
$C(\delta, \mathcal{G}_{2,6})$ are not connected.
\end{prop}

\begin{proof}
Let 
$$\delta = (35,22,32,49,42,26,34,23,32,39,41,34,46,49,32)$$
be the metric in $\mathbb{R}^{6 \choose 2}$ with coordinates
$(\delta_{12},\delta_{13},\delta_{14},
\ldots, 
\delta_{45},\delta_{46}, \delta_{56})$. 
Then $d(\delta,\mathcal{T}_6) = 
 d(\delta, \mathcal{G}_{2,6}) = 5.$ 
The set $C(\delta,\mathcal{T}_6)$ 
is the union of two disjoint polyhedra. 
One is four-dimensional and corresponds to the 
6-leaf tree with nontrivial splits 
$13|2456,134|256$ and $25|1346$ 
and the other is six-dimensional and corresponds to the 
6-leaf tree with nontrivial splits 
$14|2356,134|256$ and $56|1234$. 
In this instance, 
$C(\delta,\mathcal{T}_6)
= C(\delta, \mathcal{G}_{2,6})$.
\end{proof}

Unfortunately, many of the less than desirable properties exhibited in the ultrametric case hold for tree metrics. 
Simple modifications to the constructions for ultrametics
give analogous results for tree metrics and unrooted trees
to the results in Propositions \ref{prop:polytomy} and \ref{prop:dim}, and Corollary
\ref{cor:tritomy}.
We conclude with one such example about the 
possible dimension of the set of $l^\infty$-closest tree metrics
to a point.

\begin{prop}
\label{prop:dim}
Let $\mathcal{T} \in B(n)$. There exists $\delta \in \mathbb{R}^{n \choose 2}$ such
that $\dim(C(\delta ,\mathcal{T}_n)) = 2n - 6$ and every 
tree metric in $C(\delta ,\mathcal{T}_n)$ has $\mathcal{T}$
as a tree metric representation.
\end{prop}
\begin{proof} Let $z$ be a tree metric with
tree metric representation $\mathcal{T}$.
Let $L$ be the affine hull of the cone of tree metrics corresponding to 
$\mathcal{T}$. 
This is the linear space of dimension $2n - 3$ defined by all
of the equalities of the form $x_{ik} + x_{jl} - x_{il} - x_{kj} = 0$ where $ij|kl$ is an induced quartet of $\mathcal{T}$.

Choose $i,j,k$ and $l$ so that $ij|kl$ is a quartet of $\mathcal{T}$ . 
For $\varepsilon > 0$, let 
$\delta = z + \varepsilon(e_{ik} + e_{jl} - e_{il} - e_{kj})$. 
Since $\mathcal{T}$ is binary, if $\varepsilon$ is sufficiently small, $C(\delta ,\mathcal{T}_n) = C(\delta ,L)$. 
The type of $z$ relative to $L$ is the signed vector 
$\sigma$ where 
$\sigma_{il} = +, \sigma_{kj} = +, \sigma_{ik} = -, \sigma_{jl} = -$, and all other entries are zero.
The rank of $\sigma$ in $\mathcal{O}_L$ is three, and
thus by Theorem \ref{dimoftype}, $\dim(C(\delta ,U_n)) = (2n - 3) - 3 =  2n - 6$.
\end{proof}

\section*{Acknowledgments}
We are very grateful to Seth Sullivant for suggesting this project, for many helpful conversations,
and for his feedback on early drafts.
Daniel Bernstein was partially supported by the US National Science Foundation (DMS 0954865) and the David and Lucille Packard Foundation. Colby Long was partially supported by the Mathematical Biosciences Institute and the National Science Foundation (DMS 1440386).

\bibliography{../trees}
\bibliographystyle{plain}

\end{document}